\documentclass[10pt]{article}
\usepackage{dcolumn}
\usepackage{bm}
\usepackage{verbatim}       

\usepackage[dvips]{graphicx}

\usepackage{amsthm}
\usepackage{amssymb}

\usepackage{bm}
\usepackage{amstext}
\usepackage{amsmath}
\usepackage{amsfonts}
\usepackage{units}
\usepackage{mathrsfs}
\newcommand{\bd}{\begin{definition}}                
\newcommand{\ed}{\end{definition}}                  
\newcommand{\bc}{\begin{corollary}}                 
\newcommand{\ec}{\end{corollary}}                   
\newcommand{\bl}{\begin{lemma}}                     
\newcommand{\el}{\end{lemma}}                       
\newcommand{\bp}{\begin{proposition}}            
\newcommand{\ep}{\end{proposition}}                
\newcommand{\bere}{\begin{remark}}                  
\newcommand{\ere}{\end{remark}}                     

\newcommand{\bt}{\begin{theorem}}
\newcommand{\et}{\end{theorem}}

\newcommand{\be}{\begin{equation}}
\newcommand{\ee}{\end{equation}}

\newcommand{\bit}{\begin{itemize}}
\newcommand{\eit}{\end{itemize}}
\newtheorem{theorem}{Theorem}[section]
\newtheorem{corollary}[theorem]{Corollary}
\newtheorem{lemma}[theorem]{Lemma}
\newtheorem{proposition}[theorem]{Proposition}
\theoremstyle{definition}
\newtheorem{definition}[theorem]{Definition}
\theoremstyle{remark}
\newtheorem{remark}[theorem]{Remark}

\begin{document}

\title{Topological conditions for the representation of preorders by continuous utilities}


\author{E. Minguzzi\thanks{
Dipartimento di Matematica Applicata ``G. Sansone'', Universit\`a
degli Studi di Firenze, Via S. Marta 3,  I-50139 Firenze, Italy.
E-mail: ettore.minguzzi@unifi.it} }

\date{}

\maketitle

\begin{abstract}
We remove the Hausdorff condition from Levin's theorem on the
representation of preorders by families of continuous utilities. We
compare some alternative topological assumptions in a Levin's type
theorem, and show that they are equivalent to a Polish space
assumption.
\end{abstract}

\section{Introduction}

A topological preordered space is a triple $(E,\mathscr{T},\le)$,
where $(E,\mathscr{T})$ is a topological space endowed with a {\em
preorder} $\le$, that is, $\le$ is a reflexive and transitive
relation \cite{nachbin65}.    A function $f: E\to \mathbb{R}$ is
{\em isotone} if $x\le y$ $\Rightarrow f(x)\le f(y)$, and a {\em
utility} if it is isotone and additionally ``$x \le y$ and $y\nleq
x$ $\Rightarrow f(x)< f(y)$''.

In this work we wish to establish sufficient topological conditions
on $(E,\mathscr{T})$ for the representability of the preorder
through the family $\mathcal{U}$ of continuous utility functions
with value in $[0,1]$. That is, we look for topological conditions
that imply the validity of the following property
\[
x\le y \Leftrightarrow \forall f\in \mathcal{U}, f(x)\le f(y).
\]
Economists have long been  interested in the representation of
preorders by utility functions \cite{bridges95}. More recently, this
mathematical problem has found application in other fields  such as
spacetime physics \cite{minguzzi09c} and dynamical systems
\cite{akin10b}.

To start with, it will be convenient to recall some notions from the
theory of topological preordered spaces \cite{nachbin65}. A {\em
semiclosed preordered space} $E$ is a topological preordered space
such that, for every point $x\in E$, the increasing hull
$i(x)=\{y\in E: x\le y\}$ and the decreasing hull $d(x)=\{y: y \le
x\}$, are closed. A {\em closed preordered space} $E$ is a
topological preordered space endowed with a {\em closed preorder},
that is, the graph $G(\le)=\{(x,y): x\le y\}$ is closed in the
product topology on $E\times E$.

 Let $E$ be a topological preordered space. A subset $S\subset E$ is called {\em increasing} if $i(S)=S$ and
{\em decreasing} if $d(S)=S$, where $i(S)=\bigcup_{s\in S} i(s)$ and
analogously for $d(S)$. A subset $S\subset E$ is {\em convex} if it
is the intersection of an increasing and a decreasing set, in which
case we have $S=i(S)\cap d(S)$.

A topological preordered space $E$ is {\em convex} if for every
$x\in E$, and open set $O\ni x$, there are an open decreasing set
$U$ and an open increasing  set $V$ such that $x\in U\cap V\subset
O$. Notice that according to this terminology the statement ``the
topological preordered space $E$ is convex'' differs from the
statement ``the subset $E$ is convex'' (which is always true). The
terminology is not uniform in the literature, for instance Lawson
\cite{lawson91} calls {\em strongly order convexity} what we call
convexity. The topological preordered space $E$ is {\em locally
convex} if for every point $x\in E$, the set of convex neighborhoods
of $x$ is a base for the neighborhoods system of $x$
\cite{nachbin65}. Clearly, convexity implies local convexity.

A topological preordered space  is a {\em normally preordered space}
if it is semiclosed preordered and for every closed decreasing set
$A$ and closed increasing set $B$ which are disjoint, it is possible
to find an open decreasing set $U$ and an open increasing set $V$
which separate them, namely $A\subset U$, $B\subset V$, and $U\cap
V=\emptyset$.

A {\em regularly preordered space} is a semiclosed preordered space
such that if $x \notin B$, where $B$ is a closed increasing set,
then there is an open decreasing set $U\ni x$ and an open increasing
set $V\supset B$, such that $U\cap V=\emptyset$, and analogously, a
dual property must hold for $y \notin A$ where $A$ is a closed
decreasing set.

We have the implications: normally preordered space $\Rightarrow$
regularly preordered space $\Rightarrow$  closed preordered space
$\Rightarrow$ semiclosed preordered space.

For normally preordered spaces a natural generalization  of
Urysohn's lemma holds \cite[Theor. 1]{nachbin65}: If $A$ and $B$ are
respectively a closed decreasing set and a closed increasing set
such that $A\cap B=\emptyset$, then there is a continuous isotone
function $f: E\to [0,1]$ such that $A\subset f^{-1}(0)$ and
$B\subset f^{-1}(1)$.

A trivial and well known consequence of this fact is (take $A=d(y)$
and $B=i(x)$ with $x\nleq y$)

\begin{proposition}
Let $E$ be a normally preordered space and let $\mathcal{I}$ be the
family of continuous isotone functions with value in $[0,1]$, then
\begin{equation} \label{mbi}
x\le y \Leftrightarrow \forall f\in \mathcal{I}, f(x)\le f(y).
\end{equation}
\end{proposition}

This result almost solves our original problem but for the fact that
the family of continuous utility functions is replaced by the larger
family of continuous isotone functions. Moreover, we have  still to
identify  some topological conditions on $(E,\mathscr{T})$ in order
to guarantee that $E$ is a normally preordered space. It is worth
noting that Eq. (\ref{mbi}) is one of the two conditions which
characterize the {\em completely regularly preordered spaces}
\cite{nachbin65}.

Let us recall that a $k_\omega$-space is a topological space
characterized through the following property \cite{franklin77}:
there is a countable ({\em admissible}) sequence $K_i$  of compact
sets   such that $\bigcup_{i=1}^{\infty}K_i= E$ and for every subset
$O\subset E$, $O$ is open if and only if $O\cap K_i$ is open in
$K_i$ for every $i$ (here $E$ is not required to be Hausdorff).

Recently,  the author  proved the following results
\cite{minguzzi11f}

\begin{theorem} \label{bhp}
Every  $k_\omega$-space equipped with a closed preorder  is a
normally preordered space.
\end{theorem}

\begin{theorem} \label{bhb}
Every second countable regularly preordered space admits a countable
continuous  utility representation, that is, there is a countable
set $\{f_k\}$ of continuous utility functions $f_k:E\to [0,1]$ such
that
\[
x\le y \Leftrightarrow \forall k, f_k(x)\le f_k(y).
\]

\end{theorem}

Using the previous results we obtain the following improvement of
Levin's theorem.\footnote{These references do not consider the
representation problem but rather the existence of just one
continuous utility. Nevertheless, the argument for the existence of
the whole representation is contained at the end of the proof of
\cite[Lemma 8.3.4]{bridges95}.}$^{\!,}$\footnote{The result
\cite[Theor. 1]{evren09} should not be confused with this one, since
their definition of utility differs from our. That theorem can
instead be deduced from the stronger theorem \ref{bhp}. Also note
that in their proof they tacitly use a $k_\omega$-space assumption
which can nevertheless be justified.} \cite{levin83} \cite[Lemma
8.3.4]{bridges95}

\begin{corollary}  \label{bsc}
Let $(E,\mathscr{T},\le)$ be a  second-countable $k_\omega$-space
equipped with a closed preorder, then there is a countable family
$\{u_k\}$ of continuous utility functions $u_k: E\to [0,1]$ such
that
\[
x \le y \ \Leftrightarrow \ \forall k, \ u_k(x)\le u_k(y) .
\]
\end{corollary}

\begin{proof}
Every closed preordered $k_\omega$-space  is a normally preordered
space (Theor. \ref{bhp}). Since $E$ is a second countable regularly
preordered space it admits a countable continuous utility
representation (Theor. \ref{bhb}).
\end{proof}

With respect to the references we have removed the Hausdorff
condition.\footnote{In \cite{herden02} it was first suggested that
the Hausdorff condition could be removed. This generalization is non
trivial and  requires some care in the reformulation and
generalization of some extendibility results \cite{minguzzi11f}.}
Another interesting improvement can be found in \cite{caterino09}.
In the remainder of the work we wish to compare this result with
other reformulations which use different topological assumptions.

%
%

\subsection{Topological preliminaries}

Since in this work we do not assume Hausdorffness of $E$ it is
necessary to clarify that in our terminology a topological space is
{\em locally compact} if every point admits a compact neighborhood.


\begin{definition} \label{hem}
A topological space $(E,\mathscr{T})$ is {\em hemicompact} if there
is a countable sequence $K_i$, called {\em admissible}, of compact
sets such that
 every compact set is contained
in some $K_i$ (since points are compact we have
$\bigcup_{i=1}^{\infty}K_i= E$, and without loss of generality we
can assume $K_i \subset K_{i+1}$).
\end{definition}

The following facts are well known  (Hausdorffness is not required).
Every compact set is hemicompact and every hemicompact set is
$\sigma$-compact. Every locally compact Lindel\"of space is
hemicompact,  and every first countable hemicompact space is locally
compact.\footnote{In order to prove the last claim,  modify slightly
the proof given in \cite[p. 486]{arens46} replacing ``Suppose no
neighborhood $V_i$ has a compact closure'' with ``Suppose $x$ has no
compact neighborhood''.}${}^{,}$\footnote{A first countable
Hausdorff hemicompact $k$-space space need not be second countable.
Indeed, as stressed in \cite{franklin77} not even compactness is
sufficient as the unit square with a suitable topology provides a
counterexample \cite[p. 73]{steen70}.}
%
%
%
%
%
%

%
%
%
%

\begin{definition}
A topological space $E$ is a {\em k-space}  if for every subset
$O\subset E$, $O$ is open if and only if, for every compact set
$K\subset E$, $O\cap K$ is open in $K$.
\end{definition}

We remark that we use the definition given in \cite{willard70} and
so do not include Hausdorffness in the definition as done in
\cite[Cor. 3.3.19]{engelking89}.


Every  first countable or locally compact space is a
$k$-space.\footnote{Modify slightly the proof in \cite[Theor.
43.9]{willard70}} Thus under second countability ``hemicompact
$k$-space'' is equivalent to local compactness.


It is easy to prove that an hemicompact $k$-space is a
$k_{\omega}$-space and the converse can be proved under $T_1$
separability (see \cite[Lemma 9.3]{steenrod67}). Further, in an
hemicompact $k$-space every admissible sequence $K_i$, $K_i\subset
K_{i+1}$, in the sense of the hemicompact definition is also an
admissible sequence in the sense of the $k_\omega$-space definition.
The mentioned results imply the chain of implications
\begin{center}
compact $\Rightarrow$ hemicompact $k$-space $\Rightarrow$
$k_\omega$-space $\Rightarrow$ $\sigma$-compact $\Rightarrow$
Lindel\"of
\end{center}
and the fact that local compactness makes the last four properties
coincide.

A continuous function $f: X \to Y$ between topological spaces is
said to be  a {\em quasi-homeomorphism} if the following conditions
are satisfied \cite{kaiwing72,echi09}:
\begin{itemize}
\item[(i)] For any closed set $C$ in $X$, $f^{-1}(\overline{f(C)})= C$.
\item[(ii)] For any closed set $F$ in $Y$, $\overline{f(f^{-1}(F))}= F$.
\end{itemize}
Every quasi-homeomorphism establishes a bijective correspondence
$\psi_f: CL(Y)\to CL(X)$ between the closed sets of $Y$ and $X$
through the definition $\psi_f(C)=f^{-1}(C)$.


\begin{remark} \label{mbo}
If $f$ is surjective (ii) is satisfied. Furthermore, a quotient
(hence surjective) map which satisfies $f^{-1}(f(C))=C$ for every
closed set $C$ (or equivalently, for every open set) is a
quasi-homeomorphism. Indeed, if $C$ is closed then $f(C)$ is closed,
because of the identity $f^{-1}(f(C))=C$ and the definition of
quotient topology. Thus both properties (i)-(ii) hold, and $f$ is a
quasi-homeomorphism. The given argument also shows that $f$ is
closed  (and open). Furthermore, it can be shown that a
quasi-homeomorphism is surjective if and only if it is closed, if
and only if it is open \cite[Prop. 2.4]{echi09}.
\end{remark}

%
%
%
%
%
%
%
%
%


\section{Ordered quotient and local convexity}

On a topological preordered space $E$ the relation $\sim$, defined
by $x\sim y$ if $x\le y$ and $y\le x$, is an equivalence relation.
Let $E/\!\!\sim$ be the quotient space, $\mathscr{T}/\!\!\sim$ the
quotient topology, and let $\lesssim$ be defined by, $[x]\lesssim
[y]$ if $x\le y$ for some representatives (with some abuse of
notation we shall denote with $[x]$ both a subset of $E$ and a point
on $E/\!\!\sim$). The quotient preorder is by construction an order.
The triple $(E/\!\!\sim,\mathscr{T}/\!\!\sim,\lesssim)$ is a
topological ordered space and $\pi: E\to E/\!\!\sim$ is the
continuous quotient projection.

\begin{remark} \label{dks}
Taking into account the definition of quotient topology we have that
every open (closed) increasing  set on $E$ projects to an open
(resp. closed) increasing set on $E/\!\!\sim$ and all the latter
sets can be regarded as such projections. The same holds replacing
{\em increasing} by {\em decreasing}. As a consequence,
$(E,\mathscr{T},\le)$ is a normally preordered space (semiclosed
preordered space, regularly preordered space) if and only if
$(E/\!\!\sim,\mathscr{T}/\!\!\sim,\lesssim)$ is a normally ordered
space (resp. semiclosed ordered space, regularly ordered space). The
effect of the quotient $\pi: E\to E/\!\!\sim$ on the topological
preordered properties has been studied in \cite{kunzi05}.
\end{remark}

\begin{remark}
A set $S\subset E$ is convex if and only if $\pi(S)$ is convex.
Indeed, let $U$ and $V$ be respectively  decreasing and increasing
sets, we have $\pi(U\cap V)=\pi(U)\cap \pi(V)$ because: $U\cap
V\subset \pi^{-1}(\pi(U\cap V))\subset \pi^{-1}(\pi(U)\cap \pi(V))=
\pi^{-1}(\pi(U))\cap \pi^{-1}(\pi(V))=U\cap V $.
\end{remark}

%
%
%

\begin{proposition} \label{gvh} Let $(E,\mathscr{T},\le)$ be a topological
preordered space. If local convexity holds at $x\in E$ then $[x]$ is
compact and every open neighborhood of $x$ is also an open
neighborhood of $[x]$. If $E$ is locally convex then every open set
is saturated with respect to $\pi$ (that is $\pi^{-1}(\pi(O))=O$ for
every open set $O$). Hence $\pi$ is a (surjective)
quasi-homeomorphism, in particular $\pi$ is open, closed and proper.
\end{proposition}

\begin{proof}
Let $O$ be an open neighborhood of $x$ and let $C$ be a convex set
such that $x\in C\subset O$, then $[x]=d(x)\cap i(x)\subset d(C)\cap
i(C)=C \subset O$, thus $O$ is also an open neighborhood for $[x]$.
The compactness of $[x]$ follows easily.

Let $O\subset E$ be an open set and let $x\in O$. We have already
proved that $[x]\subset O$. Since this is true for every $x\in O$,
we have $\pi^{-1}(\pi(O))=O$. Therefore, by remark \ref{mbo}, since
$\pi$ is a quotient map it is  a quasi-homeomorphism which is open
and closed. Every such map is easily seen to be proper.
\end{proof}

%
%



\begin{remark} \label{knx}
By the previous result under local convexity the quotient $\pi$
establishes a bijection between the respective families in $E$ and
$E/\!\!\sim$ of open sets, closed sets, compact sets, increasing
sets, decreasing sets and convex sets. Continuous isotone functions
on $E$ pass to the quotient on $E/\!\!\sim$ and conversely,
continuous isotone functions on $E/\!\!\sim$ can be lifted to
continuous isotone functions on $E$. As a consequence, many
properties are shared between $E$ and $E/\!\!\sim$ regarded as
topological preordered spaces (one should not apply this observation
carelessly, otherwise one would conclude that $\le$ is an order and
that $\mathscr{T}$ is Hausdorff). For instance, we have
\end{remark}

\begin{proposition}
If $E$ is a locally convex closed preordered space then $E/\!\!\sim$
is a locally convex closed ordered space.
\end{proposition}

\begin{proof}
We just prove closure to show how the argument works. If
$[x]\not\lesssim [y]$ then $x\nleq y$. The representatives $x$ and
$y$ are separated by open sets \cite[Prop. 1, Chap. 1]{nachbin65}
$U_x$ and $U_y$ such that $i(U_x)\cap d(U_y)=\emptyset$. By local
convexity the increasing neighborhood of $x$, $i(U_x)$, projects
into an increasing neighborhood $\pi(i(U_x))$ of $[x]$. Analogously,
$\pi(d(U_y))$ is a decreasing neighborhood of $[y]$ which is
disjoint from $\pi(i(U_x))$. We conclude that $\lesssim$ is closed
\cite[Prop. 1, Chap. 1]{nachbin65}.
\end{proof}
%
%
%

The property of closure for the graph of the preorder does not pass
to the quotient without additional assumptions \cite{kunzi05}. For
instance, the previous result holds with ``locally convex'' replaced
by $k_\omega$-space \cite{minguzzi11f}.

\begin{remark}
In a topological space $(E,\mathscr{T})$ the {\em specialization
preorder} is defined by $x\preceq y$ if $\overline{x}\subset
\overline{y}$. Two points $x,y$ are indistinguishable according to
the topology if $x\preceq y$ and $y\preceq x$, denoted $x\simeq y$,
since in this case they have the same neighborhoods. The quotient
under $\simeq$ of the topological space is called {\em Kolmogorov
quotient} or {\em $T_0$-identification} and gives a $T_0$-space,
sometimes called the {\em  $T_0$-reflection} of $E$. The Kolmogorov
quotient is by construction open, closed and a quasi-homeomorphism.

The first statement of proposition \ref{gvh} implies that under
local convexity if $x\sim y$ then $x$ and $y$ have the same
neighborhoods, that is, $x\simeq y$. If the preorder $\le$ on $E$ is
semiclosed the converse holds because
$\overline{y}=\overline{x}\subset i(x)\cap d(x)$, which implies,
$y\sim x$. Thus in a locally convex semiclosed preordered space,
$\pi$ is the { Kolmogorov quotient}  and $E/\!\!\sim$ is the
$T_0$-identified space. Actually, $E/\!\!\sim$ is a $T_1$-space
because it is a semiclosed ordered space (remark \ref{dks}) thus
$\overline{[x]}\subset i_{E/\!\!\sim}([x])\cap
d_{E/\!\!\sim}([x])=\{[x]\}$. If additionally $E$ is a closed
preordered space we already know that $E/\!\!\sim$ is a closed
ordered space.

%
%
%

Another way to prove that $\lesssim$ is closed is to observe that
the $T_0$-reflection of a product is the product of the
$T_0$-reflections, that is, $\pi\times \pi$ is the Kolmogorov
quotient of $E\times E$, and since the Kolmogorov quotient is closed
it sends the closed graph $G(\le)$ into the graph $G(\lesssim)$
which is therefore closed. In summary we have proved
\end{remark}
%
%

\begin{proposition} \label{jsk}
Let $E$ be a locally convex semiclosed preordered space then $\pi:
E\to E/\!\!\sim$ is the $T_0$-identification of $E$ and $E/\!\!\sim$
is $T_1$. Furthermore, if $E$ is also a closed preordered space then
$E/\!\!\sim$ is a closed ordered space and hence $T_2$.
\end{proposition}

The next proposition will be useful (see Prop. \ref{juj}) and is an
immediate corollary of remark \ref{knx}.

\begin{proposition} \label{kyk}
If $(E,\mathscr{T},\le)$ is (locally) convex then
$(E/\!\!\sim,\mathscr{T}/\!\!\sim, \lesssim)$ is  (resp. locally)
convex. If $(E,\mathscr{T},\le)$ is locally convex and locally
compact then $(E/\!\!\sim,\mathscr{T}/\!\!\sim, \lesssim)$ is
locally compact, and if additionally $E$ is a closed preordered
space then every point of $E$ admits a base of closed compact
neighborhoods (but $E$ need not be $T_1$).
\end{proposition}

\section{Equivalence of some topological assumptions}
We wish to clarify the relative strength of some topological
conditions that can be used in a Levin's type theorem.

 Let us recall that a {\em
Polish space} is a topological space which is homeomorphic to a
separable complete metric space \cite[Part II, Chap. IX, Sect.
6]{bourbaki66}. A {\em pseudo-metric} is a metric for which the
condition $d(x,y)=0 \Rightarrow x=y$, has been dropped
\cite{kelley55}. The relation $x \thickapprox y $ if $d(x,y)=0$, is
an equivalence relation and the quotient $E/\!\thickapprox$ is a
metric space.

A pseudo-metrizable space is a topological space with a topology
which comes from some pseudo-metric. In particular, it is Hausdorff
if and only if it is metrizable because the Hausdorff property holds
if and only if the equivalence classes are trivial. We say that a
space is a pseudo-Polish space if it is homeomorphic to a
pseudo-metric space and the quotient under $\thickapprox$ is a
Polish space. Note that every pseudo-Polish space is separable.

The next result is purely topological (see Prop. \ref{jsk}) but at
some places it makes reference to a preorder. This is done because
it is meant to clarify the topological conditions underlying a
Levin's type theorem in which the presence of a closed preorder is
included in  the assumptions.

\begin{proposition} \label{juj}
Let us consider the following properties for  a topological space
$(E,\mathscr{T})$ and let $\le$ be any preorder on $E$ (e.g. the
discrete-order)
\begin{itemize}
\item[(i)] second-countable $k_\omega$-space,
\item[(ii)] second-countable locally compact,
\item[(iii)] pseudo-metrizable hemicompact $k$-space,
\item[(iv)] locally compact pseudo-Polish space.
\end{itemize}
Then $(iv)\Leftrightarrow (iii) \Rightarrow (ii) \Rightarrow (i)$.
Furthermore,  if $(E,\mathscr{T},\le)$ is a locally convex
semiclosed preordered space we have $(i) \Rightarrow (ii)$, and if
$(E,\mathscr{T},\le)$ is a locally convex closed preordered space we
have $(ii) \Rightarrow (iii)$ (note that the discrete-order is
locally convex thus the former implication holds also under $T_1$
separability of $\mathscr{T}$ and the latter implication holds also
under $T_2$ separability of $\mathscr{T}$). In particular they are
all equivalent if $(E,\mathscr{T},\le)$ is a locally convex closed
preordered space (e.g. under Hausdorffness).
\end{proposition}

\begin{proof}
We shall make extensive use of results recalled in the introduction.

(ii) $\Rightarrow$ (i). Every second countable locally compact space
is an hemicompact $k$-space and hence a $k_\omega$-space.

(i) $\Rightarrow$ (ii). Assume that $(E, \mathscr{T},\le)$ is a
locally convex semiclosed preordered space. If we prove that $E$ is
hemicompact we have finished because first countability and
hemicompactness imply local compactness. We have already proved that
$(E/\!\!\sim,\mathscr{T}/\!\!\sim)$ is $T_1$ (Prop. \ref{jsk}). But
$(E/\!\!\sim,\mathscr{T}/\!\!\sim)$ is a $k_\omega$-space by a
non-Hausdorff generalization of Morita's theorem \cite{minguzzi11f}
thus $E/\!\!\sim$ is hemicompact \cite[Lemma 9.3]{steenrod67}. Let
$\tilde{K}_i$ be an admissible sequence on $E/\!\!\sim$, since $\pi$
is proper (Prop. \ref{gvh}) the sets $K_i=\pi^{-1}(\tilde{K}_i)$ are
compact. They give an admissible sequence for the hemicompact
property, indeed if $K$ is any compact on $E$ then $\pi(K)$ is
compact on $E/\!\!\sim$ thus there is some $\tilde{K}_i$ such that
$\pi(K)\subset \tilde{K}_i$. Finally, $K\subset
\pi^{-1}(\pi(K))\subset \pi^{-1}(\tilde{K}_i)= K_i$.


(ii) $\Rightarrow$ (iii). A second countable locally compact space
is an hemicompact $k$-space.
Since $(E, \mathscr{T},\le)$ is a locally convex closed preordered
space, $E/\!\!\sim$ is Hausdorff (Prop. \ref{jsk}).
 Local convexity, local compactness,  and second countability
pass to the quotient $E/\!\!\sim$ (see Prop. \ref{gvh},\ref{kyk})
which is therefore metrizable by Urysohn's theorem. Thus $E$ is
pseudo-metrizable with the pullback by $\pi$ of the metric on
$E/\!\!\sim$.

(iii) $\Rightarrow$ (ii).  A pseudo-metrizable space is second
countable if and only if it is separable \cite[Theor. 11 Chap.
4]{kelley55} thus it suffices to prove separability. In particular,
since $E$ is $\sigma$-compact it suffices to prove separability on
each compact set $K_n$ (of the hemicompact decomposition) with the
induced topology (which comes from the induced pseudo-metric). It is
known that every compact pseudo-metrizable space is second countable
\cite[Theor. 5, Chap. 5]{kelley55} and hence separable, thus we
proved that $E$ is second countable. As first countability and the
hemicompact property imply local compactness we get the thesis.


(iv) $\Rightarrow$ (iii). $E$ is a separable pseudo-metrizable space
thus second countable \cite[Theor. 11 Chap. 4]{kelley55}. Second
countability and local compactness imply the hemicompact $k$-space
property.

(iii) $\Rightarrow$ (iv). Since (iii) $\Rightarrow$ (ii), $E$ is
second countable and locally compact. Let $d$ be a compatible
pseudo-metric on $E$ and let $E/\!\thickapprox$ be the metric
quotient. Since $\pi_{\thickapprox}: E \to E/\!\thickapprox$ is an
open continuous map (actually a quasi-homeomorphism) and $E$ is
second countable and locally compact then $E/\!\thickapprox$ is
second countable and locally compact too.
%
%
We conclude by \cite[23C]{willard70} that the one point
compactification of $E/\!\thickapprox$ is metrizable, and by
compactness the one point compactification of $E/\!\thickapprox$ is
completely metrizable. Further, since $E/\!\thickapprox$ is
separable its one point compactification is also separable. We
conclude that the one point compactification of $E/\!\thickapprox$
is Polish, and since $E/\!\thickapprox$ is Hausdorff and locally
compact, $E/\!\thickapprox$ is an open subset of a Polish space
hence Polish \cite[Part II, Chap. IX, Sect. 6]{bourbaki66}.
\end{proof}

\begin{remark}
If $E$ is a locally convex closed preordered space and (iv) holds,
as is implied by the assumptions of corollary \ref{bsc}, then the
local compactness mentioned in (iv) implies, despite the lack of
Hausdorffness, the stronger versions of local compactness (Prop.
\ref{kyk}).
\end{remark}

\section{Conclusions}
We have deduced an improved version of Levin's theorem in which the
Hausdorff condition has been removed. Furthermore, some alterative
topological assumptions underlying a Levin's type theorem have been
compared and we  we have shown that the original Levin's theorem
included a Polish space assumption.

\section*{Acknowledgments}
The material of this work was initially contained in a first
expanded version of \cite{minguzzi11f}, and it has benefited from
suggestions by an anonymous referee in connection with Prop.
\ref{jsk}.


\end{document}